\documentclass{amsart}
\usepackage[utf8]{inputenc}
\usepackage{amsmath}
\usepackage{amsthm}
\usepackage{amsfonts, dsfont}
\usepackage{amssymb}
\usepackage[all]{xy}
\usepackage{indentfirst}
\usepackage[pagebackref=true]{hyperref}

\usepackage{cleveref}
\usepackage[alphabetic, initials]{amsrefs}

\usepackage{faktor}
\usepackage{xfrac}

\newtheorem{thm}{Theorem}[section]
\newtheorem{prop}[thm]{Proposition}
\newtheorem{lem}[thm]{Lemma}
\newtheorem{theorem}[thm]{Theorem}
\newtheorem{corollary}[thm]{Corollary}
\newtheorem{proposition}[thm]{Proposition}

\newtheorem{lemma}[thm]{Lemma}
\newtheorem*{theorem*}{Theorem}

\theoremstyle{definition}

\newtheorem*{defn*}{Definiton}
\newtheorem{example}[thm]{Example}
\newtheorem{definition}[thm]{Definition}

\newtheorem{remark}[thm]{Remark}

\newcommand{\Z}{\mathbb{Z}} %% Integers
\newcommand{\C}{\mathbb{C}} %% Complex
\usepackage{verbatim} %% Enables 'comment'
\usepackage{color}
\newcommand{\G}{\Gamma}

\DeclareMathOperator{\interior}{int}

\newcommand{\La}{\Lambda}
\newcommand{\la}{\lambda}

\newcommand{\Sub}{\operatorname{Sub}}
\newcommand{\Stab}{\operatorname{Stab}}
\newcommand{\Rep}{\operatorname{Rep}}

\newcommand{\act}{\!\curvearrowright\!}

\newcommand{\supp}{\operatorname{supp}}

\newcommand{\cG }{\mathcal{G}}
\newcommand{\cH}{\mathcal{H}}
\newcommand{\cI}{\mathcal{I}}

\newcommand{\cP}{\mathcal{P}}

\newcommand{\one}{\boldsymbol{1}}

\renewcommand{\b}{\mathfrak{b}}

\newcommand{\Xoo}{X_{\hspace{-0.05em}\scriptscriptstyle 0}^{\hspace{-0.05em}\scriptscriptstyle 0}}

\newcommand{\X}{\bar{X}}

\title{Boundary maps and covariant representations}
\author{Mehrdad Kalantar}
\address{Mehrdad Kalantar\\ University of Houston\\ USA}
\email{mkalantar@uh.edu}
\author{Eduardo Scarparo}
\address{Eduardo Scarparo\\ Federal University of Santa Catarina\\ Brazil}
\email{eduardo.scarparo@posgrad.ufsc.br}

\thanks{MK is supported by a Simons Foundation Collaboration Grant (\# 713667). ES was partially supported by an ERCIM ‘Alain Bensoussan’ Fellowship}

\begin{document}
\maketitle
%%%%%%%%%%%%%%%%
\begin{abstract}
We extend applications of Furstenberg boundary theory to the study of $C^*$-algebras associated to minimal actions $\G\act X$ of discrete groups $\G$ on locally compact spaces $X$. We introduce boundary maps on $(\G,X)$-$C^*$-algebras and investigate their applications in this context. Among other results, we completely determine when $C^*$-algebras generated by covariant representations arising from stabilizer subgroups are simple.  
We also characterize the intersection property of locally compact $\G$-spaces and simplicity of their associated crossed products.
\end{abstract}

\section{Introduction}

This work is another step in the recent rapid progress 
in understanding when $C^*$-algebras associated to groups and group actions are simple (see, e.g., \cite{KK}, \cite{BKKO}, \cite{Ken},\cite{Kaw}, \cite{KenSch}, \cite{BekKal}, \cite{KalSca20}).

In many of these results, a central technique consists of investigating equivariant ucp maps from a $C^*$-algebra that one wants to understand into some $\G$-injective $C^*$-algebra with good rigidity properties.

All the progress so far in this direction has been in the framework of compact $\G$-spaces. 
In this paper, we are interested in the $C^*$-algebras generated by covariant representations of minimal actions $\G\act X$ of discrete groups $\G$ on locally compact spaces $X$.
Our key tool is the notion of boundary maps. In case $X$ is compact, we denote by $\partial(\G,X)$ the generalized Furstenberg boundary of the pair $(\G, X)$, i.e. the spectrum of the $\G$-injective envelope of $C(X)$.  If $X$ is not compact, we define $\partial(\G,X)$ in a similar way.

\begin{definition} 
Let $X$ be a locally compact $\G$-space. By a \emph{$(\G, X)$-$C^*$-algebra} we mean a pair $(A,\iota)$, where $A$ is a $C^*$-algebra endowed with an action of $\G$ and $\iota:C_0(X)\to A$ is an equivariant injective nondegenerate $*$-homomorphism.  
A \emph{$(\G, X)$-boundary map} on $A$ is a $\G$-equivariant ccp map $\psi\colon A\to C_0(\partial(\G,X))$ such that $\psi(\iota(f))=f\circ\b_X$ for every $f\in C_0(X)$, where $\b_X\colon \partial(\G,X)\to X$ is the canonical map. 
\end{definition}

Such maps always exist (Proposition~\ref{Kaw3.3}). One of our main results is the following uniqueness result for $(\G, X)$-boundary maps on the $C^*$-algebras generated by certain covariant representations of boundary actions, called \emph{germinal representation} (see Definition~\ref{grprep})

\begin{theorem*}[Theorem~\ref{thm:uniq}]
Let $X$ be a minimal locally compact $\G$-space, and $(\pi,\rho)\in\Rep(\G, X)$ a germinal representation. Then there is a unique $(\G,X)$-boundary map on the $C^*$-algebra $C^*_{\pi\times\rho}(\G,X)$ generated by the representation $(\pi,\rho)$. 
\end{theorem*}

As an application, we get the following generalization of the fact that the amenable radical of a group $\G$ is trivial if and only if its reduced $C^*$-algebra $C^*_r(\G)$ admits a unique tracial state if and only if $\G\act \partial_F\G$ is faithful (\cite{BKKO}, \cite{Fur}). These facts have also been generalized in \cite[Corollary 1.12]{Ur} in a different direction.

\begin{theorem*}[Theorem~\ref{faithful}]
Let $X$ be a minimal compact $\G$-space. The following conditions are equivalent:
\begin{enumerate}
\item[(i)] The action $\G\act\partial(\G,X)$ is faithful.
\item[(ii)] The canonical conditional expectation $\psi\colon C(X)\rtimes_r\G\to C(X)$ is the unique equivariant conditional expectation such that  $\psi(C^*_r(\G))\subset\C1_{C(X)}$.
\item[(iii)] For any $x\in X$, the stabilizer $\G_x$ does not contain any non-trivial amenable normal subgroup of $\G$.
\item[(iv)] There is $x\in X$ such that the open stabilizer $\G_x^0$ does not contain any non-trivial amenable normal subgroup of $\G$.
\end{enumerate}
\end{theorem*}

This result should be seen as a complement to the characterization by Kawabe in \cite{Kaw} of topological freeness of the action $\G\act \partial(\G,X)$ in terms of properties of the crossed product $C(X)\rtimes_r\G$ and in terms of the dynamical behavior of the stabilizer subgroups.

Given a boundary action $\G\act X$, the authors investigated in \cite{KalSca20} properties of $C^*$-algebras generated by quasi-regular representations arising form stabilizers of the action. In the more general case in which $\G\act X$ is just a minimal action, many of the techniques of \cite{KalSca20} can still be applied, but the conclusions hold for \emph{covariant representations} of $(\G,X)$ instead of \emph{unitary representations} of $\G$.

In this direction, our main result is the following complete characterization of simplicity of $C^*$-algebras of such representations. Given an action on a locally compact space $\G\act X$ and $x\in X$, let $\rho_x\colon C_0(X)\to B(\ell^2(\G x))$ be the map defined by $\rho_x(f)(\delta_{y}):=f(y)\delta_y$ and $\la_{\G/\G_x}\colon\G\to B(\ell^2(\G x))$ be the unitary representation determined by $\la_{\G/\G_x}(g)(\delta_y):=\delta_{gy}$.

\begin{theorem*}[Theorem~\ref{thm:gk}]
Let $X$ be a locally compact space and $\G\act X$ be a minimal action. Given $x\in X$, the $C^*$-algebra $C^*_{\la_{\G/\G_x}\times\rho_x}(\G,X)$ is simple if and only if $\frac{\G_x}{\G_x^0}$ is amenable.
\end{theorem*}

We also investigate the intersection property and simplicity of crossed products associated to actions on locally compact spaces. In this direction, our main result is the following, which is based upon and generalizes Kawabe's criterion for simplicity of crossed products associated to minimal compact $\G$-spaces.
\begin{theorem*}[Theorem~\ref{thm:scp}]
Let $X$ be a minimal locally compact $\G$-space. The following conditions are equivalent:
\begin{enumerate}
\item[(i)] The reduced crossed product $C_0(X)\rtimes_r\G$ is simple.
\item[(ii)] For every $x,z\in X$ and every amenable subgroup $\La$ of $\G_x$, there exists a net $(g_i)\subset\G$ such that $g_i\La g_i^{-1}\to\{e\}$ and $g_ix\to z$. 
\item[(iii)]There exists $x\in X$ such that for every amenable subgroup $\La$ of $\G_x$, there exists a net $(g_i)\subset\G$ such that $g_i\La g_i^{-1}\to\{e\}$ and $g_i x$ converges to a point in $X$.
\end{enumerate}
\end{theorem*}

Given a topologically transitive compact $\G$-space $X$,  in Theorem \ref{thm:genna} we give a dynamical characterization of $\partial(\G,X)$. This generalizes the work of Naghavi \cite{Nag}, who provided such a characterization for compact minimal $\G$-spaces. Even if one is only interested in the case of minimal actions, topological transitivity shows up naturally when one considers non-compact minimal spaces, since the one-point compactification is topologically transitive and non-minimal.  In particular, in Theorem \ref{thm:bnd-transit} we use Theorem \ref{thm:genna} for describing $\partial(\G,X)$ when $X$ is a discrete transitive $\G$-space.

This work is organized as follows. In Section 2, we collect some preliminaries about topological dynamics and covariant representations.  In Section 3, we investigate the notions of boundaries and boundary maps. We also generalize the notion of Furstenberg boundary of a $\G$-space $X$ to the case in which $X$ is locally compact.  In Section 4, we investigate boundary maps in more detail, specially aiming for uniqueness results. In Section 5, we determine conditions for simplicity of certain $C^*$-algebras associated to minimal actions.  Finally, in Section 6 we investigate the intersection property and simplicity of crossed products associated to actions on locally compact spaces.

\textbf{Acknowledgements.}
We thank the referee for their careful reading of the paper and their helpful comments.

\section{Preliminaries}

\subsection{Topological dynamics}
Throughout the paper $\G$ is a discrete group, and $\G\act X$ denotes an action of $\G$ by homeomorphisms on a locally compact Hausdorff space $X$. In this case we say $X$ is a \emph{locally compact $\G$-space}.  If $X$ is not compact, we denote by $\X$ its one-point compactification, and regard it as a $\G$-space in a natural way. If $X$ is compact, then we let $\X:=X$.

Given a locally compact $\G$-space $X$ and $x\in X$, the \emph{stabilizer subgroup} of $x$ is $\G_x:=\{g\in\G:gx=g\}$; and the \emph{open stabilizer} is the subgroup $$\G_x^0:=\{g\in \G : \text{ $g$ fixes an open neighborhood of $x$}\}.$$ 
Notice that $\G_x^0$ is a normal subgroup of $\G_x$.

We denote by $X^g$ the fixed set of $g\in\G$, that is, the set $X^g:=\{x\in X : gx=x\}$. Note that $g\in \G_x$ if and only if $x\in X^g$, and $g\in \G_x^0$ if and only if $x\in \interior X^g$.

We recall the following standard terminology for a given action $\G\act X$; we say
\begin{itemize}
\item
the action is \emph{minimal}, or $X$ is a minimal $\G$-space, if $X$ has no non-empty proper closed $\G$-invariant subsets;

\item
the action is \emph{topologically transitive}, or $X$ is a topologically transitive $\G$-space, if there is $x\in X$ with dense orbit;

\item
the action is \emph{topologically free}, or $X$ is a topologically free $\G$-space, if $\interior X^g =\emptyset$ for every non-trivial $g\in \G$;
\item
the action is \emph{faithful}, or $X$ is a faithful $\G$-space, if $X^g \neq X$ for every non-trivial $g\in \G$.
\end{itemize}

Denote by $\Sub(\G)$ the set of subgroups of $\G$, endowed with the Chabauty topology; this is the restriction to $\Sub(\G)$ of the product topology on $\{0, 1\}^\G$, where 
every subgroup $\La\in \Sub(G)$ is identified with its characteristic function $\one_\La \in \{0, 1\}^\G$. Given a locally compact $\G$-space $ X$, let $\Xoo\subset X$ be the set of points at which the map $\Stab^0: X \to \Sub(\G)$ given by $x\mapsto \G_x^0$ is continuous. If $\G$ is countable, then $\Xoo$ is dense in $X$ (\cite[Proposition 2.1]{KalSca20} or,  for a similar statement,  \cite{Vor}, \cite{GW}).

\subsection{Covariant representations}

Given a locally compact $\G$-space $X$, a \emph{nondegenerate covariant representation} of $(\G,X)$ is a pair $(\pi,\rho)$, where $\pi$ is a unitary representation of $\G$ and $\rho\colon C_0(X)\to B(\cH_\pi)$ is a $\G$-equivariant $*$-homomorphism, which is nondegenerate in the sense that $\overline{\mathrm{span}}\{\rho(f)\xi:f\in C_0(X),\xi\in\cH_\pi\}=\cH_\pi$. In this case, we let $C^*_{\pi\times\rho}(\G,X):=\overline{\mathrm{span}}\{\rho(f)\pi(g):f\in C_0(X),g\in\G\}$. We denote by $\Rep(\G,X)$ the family of all nondegenerate covariant representation of $(\G, X)$.

Given $(\pi_1,\rho_1), (\pi_2,\rho_2) \in \Rep(\G,X)$, we say that $(\pi_1,\rho_1)$ is \emph{weakly contained} in $(\pi_2,\rho_2)$, and write $(\pi_1,\rho_1)\prec(\pi_2,\rho_2)$ if there is a $*$-homomorphism $C^*_{\pi_2\times\rho_2}(\G,X)\to C^*_{\pi_1\times\rho_1}(\G,X)$ which maps $\rho_2(f)\pi_2(g)$ to $\rho_1(f)\pi_1(g)$, for every $g\in\G$ and $f\in C_0(X)$.

We are particularly interested in the following class of representations. Given $x\in X$ and $H\leq\G_x$ there is a canonical covariant representation $(\la_{\G/H}, \rho_x)\in\Rep(\G,X)$ on $B(\ell^2(\G/H))$ defined by 
\begin{equation}\label{Poisson}
\rho_x(f)(\delta_{gH}) = f(gx)\delta_{gH}
\end{equation}
 and $\la_{\G/H}(k)(\delta_{gH}) = \delta_{kgH}$ for $f\in C_0(X)$ and $g,k\in\G$. 

\begin{example}
Let $X$ be a locally compact $\G$-space and $x\in X$ a point with dense orbit. Then $C_0(X)\rtimes_r\G=C^*_{\la_\G\times\rho_x}(\G,X)$. Indeed, let $E\colon B(\ell^2(\G))\to \ell^\infty(\G)$ be the canonical conditional expectation. Given $g\in \G$ and $f\in C_0(X)$, we have that 
\begin{equation*}
E(\rho_x(f)\la_\G(g))=
\begin{cases}
0, &\text{if $g\in\G\setminus\{e\}$} \\
\rho_x(f), &\text{if $g=e$} .
\end{cases}
\end{equation*}
Since $x$ has dense orbit, we also have that $\rho_x$ is injective. From the fact that $E$ is faithful, we conclude that $C_0(X)\rtimes_r\G=C^*_{\la_\G\times\rho_x}(\G,X)$ (see e.g.  \cite[Proposition 19.8]{Exel} for how this conclusion follows from faithfulness of $E$).
\end{example}

\begin{proposition}\label{prop:prec}
Let $X$ be a locally compact $\G$-space. Given $x\in X$ and $\La_1\leq\La_2\leq\G_x$, let, for $i=1,2$, $\rho_x^i\colon C_0(X)\to B(\ell^2(\G/\La_i)$ be defined as in \eqref{Poisson}. Then $\La_1$ is co-amenable in $\La_2$ if and only if $(\la_{\G/\La_2},\rho^2_x)\prec(\la_{\G/\La_1},\rho^1_x)$.
\end{proposition}
\begin{proof}
Suppose $(\la_{\G/\La_2},\rho^2_x)\prec(\la_{\G/\La_1},\rho^1_x)$.  We claim that, in this case, $\la_{\G/\La_2}\prec\la_{\G/\La_1}$.  Indeed, fix $(u_i)\subset C_0(X)$ an approximate unit. Given $a\in\C\G$, we have
$$\|\la_{\G/\La_2}(a)\|=\sup_i\|\rho_x^2(u_i)\la_{\G/\La_2}(a)\|\leq \sup_i\|\rho_x^1(u_i)\la_{\G/\La_1}(a)\|=\|\la_{\G/\La_1}(a)\|.$$
Therefore, $\la_{\G/\La_2}\prec\la_{\G/\La_1}$ which implies that $\La_1$ is co-amenable in $\La_2$.

Conversely, assume that $\La_1$ is co-amenable in $\La_2$. Let $P\in B(\ell^2(\G/\La_1))$ be the orthogonal projection on $\ell^2(\La_2/\La_1)$. 

Given $f\in C_0(X)$ and $g\in\G$, we have that 
\begin{align*}
P\rho_x^1(f)\la_{\G/\La_1}(g)P=
\begin{cases}
f(x)\la_{\La_2/\La_1}(g), & \text{if $g\in\La_2$}\\
0, & \text{if $g\notin\La_2$.}\end{cases}
\end{align*}
Let $\psi\colon C^*_{\la_{\G/\La_1}\times\rho^1_x}(\G,X)\to C^*_{\la_{\La_2/\La_1}}(\La_2)$ be given by $\psi(a)=PaP$.
By co-amenability, there is a state $\varphi \colon  C^*_{\la_{\La_2/\La_1}}(\La_2)\to\C$ such that $\varphi(\la_{\La_2/\La_1}(g))=1$ for every $g\in \La_2$. It is easy to see that the GNS representation associated with $\varphi\circ\psi$ is unitarily equivalent to $(\la_{\G/\La_2},\rho^2_x)$.
\end{proof}

\begin{definition}\label{grprep}
Let $X$ be a locally compact $\G$-space. A \emph{germinal representation} of $(\G,X)$ is a nondegenerate covariant representation $(\pi,\rho)$ of $(\G,X)$ such that 
\begin{equation*}
\pi(g)\rho(f)=\rho(f),\quad \forall g\in\G, ~ f\in C_0(X) \text{ with } \supp f\subset  X^g .
\end{equation*}
\end{definition}

\begin{example}\label{ex:grprep}{\cite[Proposition 4.3.(ii)]{KalSca20}}
Let $X$ be a locally compact $\G$-space. Given $x\in X$ and $H\in\Sub(\G)$ such that $\G_x^0\leq H\leq \G_x$, we have that $(\la_{\G/H},\rho_x)$ is a germinal representation.
\end{example}

The following fact is immediate from the definitions.
\begin{lem}\label{lem:wcont-grprep}
Let $X$ be a locally compact $\G$-space, and let $(\pi, \rho) \in \Rep(\G,X)$ be a germinal representation. Then any covariant representation of $(\G, X)$ that is weakly contained in $(\pi,\rho)$ is also a germinal representation.
\end{lem}

\section{Boundaries of $\G$-spaces}

Let $X$ be a compact $\G$-space. We denote by $\partial(\G,X)$ the spectrum of the $\G$-injective envelope $\cI_\G(C(X))$ of $C(X)$, and we denote by $\b_X: \partial(\G,X)\to X$ the canonical continuous equivariant map. 
When $X=\{*\}$ is trivial, $\partial(\G,X)$ is the \emph{Furstenberg boundary} $\partial_F\G$ of $\G$ \cite[Theorem 3.11]{KK}.

If $X$ is a locally compact non-compact $\G$-space, recall that we denote by $\X$ its one-point compactification. In this case,  let $\partial(\G,X):=\b_{\X}^{-1}(X)$, and $\b_X:=\b_{\X}|_{\partial(\G,X)}\colon \partial(\G,X)\to X$. Notice that $\partial(\G,X)$ is a $\G$-invariant open subset of $\partial(\G,\X)$. In particular, $C_0(\partial(\G,X)) $ is an ideal of $C(\partial(\G,\X))$. Moreover,  since $\partial(\G,\X)$ is compact, one can readily check that $\b_X$ is a proper map. In particular, its transposition gives an embedding $\widetilde{\b_X}\colon C_0(X)\to C_0(\partial(\G,X))$.

A $C^*$-algebra $A$ is called a \emph{$\G$-$C^*$-algebra} if there is an action $\G\act A$ by $*$-automorphisms. 

\begin{definition}
Let $X$ be a locally compact $\G$-space. By a \emph{$(\G,X)$-$C^*$-algebra}, we mean a pair $(A,\iota)$, where $A$ is a $\G$-$C^*$-algebra and $\iota\colon C_0(X)\to A$ is an equivariant injective $*$-homomorphism, which is nondegenerate in the sense that $\overline{\mathrm{span}}\{\iota(f)a:f\in C_0(X),a\in A\}=A$. A \emph{$(\G,X)$-boundary map} on $A$ is a ccp equivariant map $\psi\colon A \to C_0(\partial(\G,X))$ such that $\psi(\iota(f))=f\circ\b_X$ for every $f\in C_0(X)$.
\end{definition}

We will often use the fact that if $\psi$ is a $(\G,X)$-boundary map on a $(\G,X)$-$C^*$-algebra $(A,\iota)$, then $\iota(C_0(X))$ is contained in the multiplicative domain of $\psi$.
\begin{example}
If $X$ is a locally compact $\G$-space, then the canonical conditional expectation $E\colon C_0(X)\rtimes_r\G\to C_0(X)$ is a $(\G,X)$-boundary map.
\end{example}
The following result extends some observations from \cite{Kaw} about $\partial(\G,X)$ to the non-compact case. 
\begin{proposition}\label{Kaw3.3}
Let $X$ be a locally compact $\G$-space.  Then one has the following:
\begin{enumerate}
\item[(i)] The space $\partial(\G,X)$ is extremally disconnected.
\item[(ii)] For any closed $\G$-invariant  set $Z\subset\partial(\G,X)$ such that $\b_X(Z)=X$, we have $Z=\partial(\G,X)$.
\item[(iii)] For any $y\in\partial(\G,X)$, the group $\G_y$ is amenable and $\G_y=\G_y^0$.
\item[(iv)] Every $(\G,X)$-$C^*$-algebra admits a $(\G,X)$-boundary map.
\item[(v)] The space $\partial(\G,X)$ is dense in $\partial(\G,\X)$.
\item[(vi)] The map $\b_X$ is closed. In particular, if $X$ is minimal, then $\partial(\G,X)$ is also minimal.
\end{enumerate}
\end{proposition}
\begin{proof}
If $X$ is compact, then (i), (ii) and (iii) were shown in \cite[Proposition 3.3]{Kaw}, whereas (iv) is a consequence of $\G$-injectivity of $C(\partial(\G,X))$. Moreover,  in the compact case, (v) and (vi) are trivial.  

Assume in the rest of the proof that $X$ is not compact.

Items (i), (ii) and (iii) are immediate consequences of the compact case.

(iv).  Let $A$ be a $(\G,X)$-$C^*$-algebra. Then the unitization $\tilde{A}$ is a $(\G,\X)$-$C^*$-algebra in a natural way.  In particular, it admits a $(\G,\X)$-boundary map $\psi$. Using an approximate unit for $C_0(X)$, one can readily check that $\psi(A)\subset C_0(\partial(\G,X))$, hence $\psi|_A$ is a $(\G,X)$-boundary map on $A$.

(v). Since $X$ is not compact, we have that $\partial(\G,X)$ is not closed in $\partial(\G,\X)$ and $Z:=\overline{\partial(\G,X)}$ is a $\G$-invariant closed subset of $\partial(\G,\X)$ such that $\b_{\X}(Z)=\X$. From (ii), we conclude that $Z=\partial(\G,\X)$.

(vi). Let $Z$ be a closed subset of $\partial(\G,X)$. Then $W:=Z\cup \b_{\X}^{-1}(\X\setminus X)$ is a compact set.  Hence, $\b_X(Z)=\b_X(W)\cap X$ is closed.
\end{proof}

\subsection*{A dynamical characterization of boundaries} 

In the remainder of this section, we generalize some results from \cite{Nag}. These results will not be used in the rest of the paper.

Given a compact $\G$-space $X$, let $\cP(X)$ be the space of regular probability measures on $X$.
Endowed with the weak* topology and the natural $\G$-action, $\cP(X)$ is a compact $\G$-space. We will often identify $X$ with the space of point mass measures $\{\delta_x\in\cP(X):x\in X\}$.  

A \emph{$\G$-extension} of $X$ is a pair $(Y,\varphi)$, where $Y$ is a compact $\G$-space and $\varphi\colon Y\to X$ is a surjective continuous equivariant map. In this case, let $\tilde{\varphi}\colon C(X)\to C(Y)$ be the embedding given by transposition.  Recall that $(C(Y),\tilde{\varphi})$ is said to be a \emph{$\G$-essential extension} of $C(X)$ if for each unital $\G$-$C^*$-algebra $A$ and each ucp equivariant map $\theta\colon C(Y)\to A$ such that $\theta\tilde{\varphi}$ is isometric, we have that $\theta$ is isometric. % (by embedding $A$ in its $\G$-injective envelope,  one can assume that $A$ is $\G$-injective).

The following result generalizes some of the results in \cite[Section 3]{Nag}.
\begin{theorem}\label{thm:genna}
Let $X$ be a topologically transitive compact $\G$-space and $(Y,\varphi)$ a $\G$-extension of $X$. The following conditions are equivalent:
\begin{enumerate}
\item[(i)]
$(C(Y),\tilde{\varphi})$ is a $\G$-essential extension of $C(X)$.

\item[(ii)]
Given $\nu\in\cP(Y)$ such that $\varphi_*\nu=\delta_x$ and $\overline{\G x}=X$, we have that $Y\subset\overline{\G\nu}$.

\item[(iii)]
There is a surjective continuous equivariant map $\alpha\colon \partial(\G,X)\to Y$ such that $\varphi\alpha=\b_X$. 
\end{enumerate}
\end{theorem}
\begin{proof}
(i)$\implies$(ii). This follows immediately from \cite[Théorème I.2]{Az}.

(ii)$\implies(iii)$.  By $\G$-injectivity, there is a ucp equivariant map $\pi\colon C(Y)\to C(\partial(\G,X))$ such that $\pi\tilde{\varphi}=\widetilde{\b_X}$.  Let $\tilde{\pi}\colon\cP(\partial(\G,X))\to\cP(Y)$ be given by transposition. Take $x\in X$ with dense orbit and $z\in\b_X^{-1}(x)$.  In particular, also $z$ has dense orbit.

Note that $\varphi_*\tilde{\pi}(\delta_z)=\delta_x$, hence $Y\subset\overline{\G\tilde{\pi}(\delta_z)}=\tilde{\pi}(\partial(\G,X))$. Let $Z:=\tilde{\pi}^{-1}(Y)\cap\partial(\G,X)$.  We claim that $\b_X(Z)=X$. Indeed, given $u\in X$, take $y\in\varphi^{-1}(u)$ and $w\in Z$ such that $\tilde{\pi}(w)=y$. Then $\b_X(w)=\varphi\tilde{\pi}(w)=u$. Therefore, $\b_X(Z)=X$. By Proposition~\ref{Kaw3.3}, we obtain that $Z=\partial(\G,X)$, hence $\tilde{\pi}(\partial(\G,X))=Y$.  Then $\alpha:=\tilde{\pi}|_{\partial(\G,X)}$ is the desired map. 

(iii)$\implies$(i).
Let $A$ be a $\G$-$C^*$-algebra and $\psi\colon C(Y)\to A$ be a ucp equivariant map such that $\psi\tilde{\varphi}$ is isometric, and we want to conclude that $\psi$ is isometric. 
By embedding $A$ in its $\G$-injective envelope, we may further assume, without loss of generality, that $A$ is $\G$-injective.
Let $\beta\colon C(\partial(\G,X))\to A$ be a ucp equivariant map such that $\beta\tilde{\alpha}=\psi$. Then $\beta\tilde{\alpha}\tilde{\varphi}$ is isometric, hence $\beta$ is isometric and so is $\psi$.
\end{proof}

\begin{definition}
Let $X$ be a topologically transitive compact $\G$-space.  A $\G$-extension $(Y,\varphi)$ of $X$ which satisfies the conditions in Theorem~\ref{thm:genna} is called a \emph{$(\G,X)$-boundary}.
\end{definition}

\begin{prop}\label{prop:trans-StonCeck->bnd}
Let $X$ be a topologically transitive non-compact $\G$-space, and let $\varphi\colon\beta X\to \X$ be the canonical map from the Stone-Čech compactification $\beta X$ to the one-point compactification $\X$. Then $(\beta X,\varphi)$ is a $(\G,\X)$-boundary. 
\end{prop}
\begin{proof}
Indeed, take $x\in X$ a point with dense orbit, and $\nu\in\cP(\beta X)$ such that $\varphi_*\nu=\delta_x$. Since $\varphi^{-1}(x)=\{x\}$, we have $\nu(\{x\})=(\varphi_*\nu)(\{x\})=1$. Therefore, $\nu=\delta_x$, hence $\beta X\subset\overline{\G \nu}$.
\end{proof}

\begin{remark}
Let $X$ be a topologically transitive locally compact $\G$-space, let $Y$ be a compactification of $X$, and let $\varphi_Y\colon\beta X\to Y$ be the continuous extension of the embedding $X\hookrightarrow Y$. Say $Y$ is a $\G$-compactification if $\varphi_Y^*(C(Y))$ is $\G$-invariant in $C(\beta X)$. In this case $Y$ turns into a compact $\G$-space and $\varphi_Y$ an equivariant map.

Now assume $Y$ is a $\G$-compactification of $X$ such that $X$ embeds as an open subset of $Y$. Then the canonical continuous surjections $\beta X\to Y \to \X$ yield inclusions of $\G$-$C^*$-algebras $C(\X)\to C(Y)\to C(\beta X)$. This implies $\partial (\G, Y) = \partial (\G, \X)$. The latter identification is canonical, thus we see that in our construction of $\partial (\G, \X)$ we may in fact choose instead of the one-point compactification, any $\G$-compactification of $X$ that contains $X$ as an open subset.
\end{remark}

\begin{example}\label{ex:amenable}
Let $X:=\G/\La$ where $\La\leq \G$ is an amenable subgroup, and consider the canonical action $\G\act X$. Then it follows from Proposition~\ref{prop:trans-StonCeck->bnd} that the $\G$-$C^*$-algebra $\ell^\infty(\G/\La)\simeq C(\beta(\G/\La))$ is a $\G$-essential extension of $C(\X)$. Since $\La$ is amenable, $\ell^\infty(\G/\La)$ is $\G$-injective, which implies that $\beta(\G/\La)=\partial(\G,\X)$ and $\partial(\G,\G/\La)=\G/\La$. 
\end{example}

The above example is in fact a special case of the following complete description of boundaries of transitive actions on discrete spaces, which in particular generalizes one of the main results of \cite{Nag} where the case of transitive actions on finite spaces was considered.

\begin{thm}\label{thm:bnd-transit}
Let $X=\G/\La$ where $\La\leq \G$ is a subgroup and consider the canonical action $\G\act X$. Then $\partial(\G,X)\simeq X\times\partial_F\La$, where the latter is endowed with the induced action $\G\act(X\times\partial_F\La)$. 
Furthermore, we have $\partial(\G,\X)\simeq\beta(X\times\partial_F\La)$.
\end{thm}

%%%%
%%%%
Let us first recall the notion of induced actions.
Let $\La$ be a subgroup of a group $\G$. 
Fix a cross-section $\pi\colon\G/\La\to\G$ such that $\pi(\La)=e$, and let $h\colon \G/\La\times\G\to\La$ be the associated cocycle given by $h(x,g)=\pi(gx)^{-1}g\pi(x)$, for $x\in\G/\La$ and $g\in\G$.

Given a $\La$-operator system $V$, the induced $\G$-action on $\ell^\infty(\G/\La,V)$ is defined in the following way: given $g\in\G$ and $f\in \ell^\infty(\G/\La,V)$, let $(gf)(x):=h(g^{-1}x,g)f(g^{-1}x)$, for $x\in\G/\La$. Using properties of the cocycle $h$, one readily checks that this is well-defined.

The following result generalizes \cite[Lemma 2.2]{Ham85}.
\begin{lem}\label{lem:inj}
If $V$ is a $\La$-injective operator system, then $\ell^\infty(\G/\La,V)$ is $\G$-injective.
\end{lem}

\begin{proof}
Let $W\subseteq \tilde W$ be an inclusion of $\G$-operator systems, and let $\phi\colon W\to \ell^\infty(\G/\La,V)$ be a $\G$-equivariant ucp map. Let $E\colon \ell^\infty(\G/\La,V)\to V$ be the map given by $E(f):=f(\La)$, for $f\in\ell^\infty(\G/\La,V)$. Then $E$ is $\La$-equivariant, and therefore so is $\psi:= E\circ \phi\colon W\to V$. By $\La$-injectivity, $\psi$ extends to a $\La$-equivariant ucp map $\tilde{\psi}\colon \tilde W\to V$.  Recall that $\pi\colon{\G/\La}\to\G$ is a fixed cross-section and define $\tilde{\phi}\colon\tilde W\to \ell^\infty(\G/\La,V)$ by 
\begin{equation*}\label{eq:ext}
\qquad\qquad\tilde{\phi}(\tilde w)(x):=\tilde\psi(\pi(x)^{-1}\tilde w) \qquad \text{for}\quad \tilde w\in \tilde W \text{ and } x\in\G/\La .
\end{equation*}
We claim that $\tilde{\phi}$ is $\G$-equivariant. Indeed, for $g\in\G$, $\tilde w\in \tilde W$ and $x\in\G/\La$, we have 
\begin{align*}
(g\tilde{\phi}(\tilde w))(x)&=h(g^{-1}x,g)\tilde{\phi}(\tilde w)(g^{-1}x)
=h(g^{-1}x,g)\tilde\psi(\pi(g^{-1}x)^{-1}\tilde w)
\\&=\tilde\psi(h(g^{-1}x,g)\pi(g^{-1}x)^{-1}\tilde w)
=\tilde\psi(\pi(x)^{-1}g\tilde w)
=\tilde{\phi}(g\tilde w)(x) ,
\end{align*}
and the claim follows. Given $w\in W$ and $x\in\G/\La$, we have
\begin{align*}
\tilde{\phi}(w)(x)&=E(\phi(\pi(x)^{-1}w))
=(\pi(x)^{-1}\phi(w))(\La)
\\&= h(\pi(x)\La,\pi(x)^{-1}) \phi(w)(\pi(x)\La)
= \phi(w)(x) ,
\end{align*}
which shows that $\tilde\phi|_{W}=\phi$. We conclude $\ell^\infty(\G/\La,V)$ is $\G$-injective.
\end{proof}

\noindent
{\it Proof of Theorem~\ref{thm:bnd-transit}.}\
Note that $\ell^{\infty}(X,C(\partial_F\La))\simeq C_b(X\times\partial_F\La)\simeq C(\beta(X\times\partial_F\La))$, and the induced action $\G\act(X\times\partial_F\La)$ is given by $g(x,z):=(gx,h(x,g)z)$, for $g\in\G$, $x\in X$ and $z\in\partial_F\La$.

Let $\varphi\colon\beta(X\times\partial_F\La)\to\X$ be the canonical continuous map which extends the projection $X\times\partial_F\La\to X$.  We claim that $\varphi^{-1}(X)=X\times\partial_F\La$.  Indeed, given $y\in\beta(X\times\partial_F\La)$ such that $\varphi(y)=x\in X$, we have $y=\lim(x_i,z_i)$ for a net $(x_i,z_i)\subset  X\times\partial_F\La$ such that $x_i\to x$.  By compactness of $\partial_F\La$,  we may assume that $z_i\to z\in\partial_F\La$. Hence, $y=(x,z)\in X\times\partial_F\La$.

We will now use Theorem \ref{thm:genna} to show that $(\beta(X\times\partial_F\La),\varphi)$ is a $(\G,\X)$-boundary. 

Given $x\in X$ and $\nu\in\cP(\beta(X\times\partial_F\La))$ such that $\varphi_*\nu=\delta_x$,  we will show that $\overline{\G\nu}=\beta(X\times\partial_F\La)$. Since $\G$ acts on $X$ transitively, we can assume that $x=\La$. 

We have that $\supp\nu\subset\varphi^{-1}(\La)=\{\La\}\times\partial_F\La$.  Since for every $\mu\in\cP(\partial_F\La)$ the weak* closure of the $\La$-orbit of $\mu$ in $\cP(\partial_F\La)$ contains all point-measures,  it follows that $\overline{\G\nu}\supset X\times\partial_F\La$,  hence $\overline{\G\nu}\supset \beta(X\times\partial_F\La)$. From Theorem \ref{thm:genna},  we obtain that $(\beta(X\times\partial_F\La),\varphi)$ is a $(\G,\X)$-boundary. 

Since, by Lemma~\ref{lem:inj}, we know that $C(\beta(X\times\partial_F\La))$ is $\G$-injective, we conclude that $\beta(X\times\partial_F\La)\simeq\partial(\G,\X)$ and $\partial(\G,X)\simeq X\times\partial_F\La$.

%%%
%%%

%%%%%%%%%%%%%%%%%%%%%%
%%%%%%%%%%%%%%%%%%%%%%

\section{Uniqueness of boundary maps}\label{sec:bdm}

Given a locally compact $\G$-space $X$ and $(\pi,\rho)\in\Rep(\G,X)$ such that $\rho$ is injective, notice that $C^*_{\pi\times\rho}(\G,X)$ is a $(\G,X)$-$C^*$-algebra. In this case, we let $\tilde{\rho}\colon C(\X)\to B(\cH_\pi)$ be the unital extension of $\rho$ (hence, if $X$ is compact, $\tilde{\rho}=\rho$).

\begin{proposition}\label{prop:ext}
Let $X$ be a locally compact $\G$-space and $(\pi,\rho)\in\Rep(\G,X)$ such that $\rho$ is injective.  Then any $(\G,X)$-boundary map on $C^*_{\pi\times\rho}(\G,X)$ admits a unique extension to a $(\G,\X)$-boundary map on $C^*_{\pi\times\tilde{\rho}}(\G,\X)$.
\end{proposition}
\begin{proof}
If $X$ is compact then there is nothing to prove. So, assume $X$ is non-compact.
Let us first show that there exists an extension. Let $\psi$ be a $(\G,X)$-boundary map on $A:=C^*_{\pi\times\rho}(\G,X)$. By \cite[Proposition 2.2.1]{BO}, we can extend $\psi$ to a ucp equivariant map $\tilde{\psi}$ from the unitization $\tilde{A}$ to $C(\partial(\G,\X))$. Using $\G$-injectivitiy of $C(\partial(\G,\X))$, we can extend $\tilde{\psi}$ to a $(\G,\X)$-boundary map on $C^*_{\pi\times\tilde{\rho}}(\G,\X)$.

Now let us prove uniqueness. Let $\varphi$ be a $(\G,\X)$-boundary map on $C^*_{\pi\times\tilde{\rho}}(\G,\X)$ which extends $\psi$. Given $y\in\partial(\G,X)$, take $f\in C_0(X)$ such that $f(\b_X(y))=1$. Given $g\in\G$, we have 
$$\varphi(\pi(g))(y)=(f\circ\b_X)(y)\varphi(\pi(g))(y)=\varphi(\rho(f)\pi(g))(y)=\psi(\rho(f)\pi(g))(y).$$
Since $\partial(\G,X)$ is dense in $\partial(\G,\X)$ (Proposition~\ref{Kaw3.3}), we conclude that $\varphi$ is uniquely determined by $\psi$.
\end{proof}

Let $X$ be a locally compact $\G$-space and $A$ a $(\G,X)$-$C^*$-algebra. Given a $(\G,X)$-boundary map on $A$, let $I_\psi:=\{a\in A:\psi(a^*a)=0\}$. It follows from the Cauchy-Schwartz inequality that $I_\psi$ is a left ideal of $A$.

The next two results give a connection between the ideal structure of a $(\G,X)$-$C^*$-algebra and $(\G,X)$-boundaries map on it.

\begin{proposition}\label{prop:ideal1}
Let $X$ be a minimal locally compact $\G$-space and $A$ a $(\G,X)$-$C^*$-algebra. For every every proper ideal $J\unlhd A$, there is a boundary map $\psi$ on $A$ such that $J\subset I_\psi$.  

In particular, if $A$ admits a unique boundary map $\psi$, then $I_\psi$ contains all proper ideals of $A$.
\end{proposition}
\begin{proof}
Given a proper ideal $J\unlhd A$, let $q\colon A\to A/J$ be the canonical quotient map. Clearly, ${A}/{J}$ is also a $(\G,X)$-$C^*$-algebra since $\G\act X$ is minimal. Let $\varphi$ be a $(\G,X)$-boundary map on ${A}/{J}$. Then $\psi:=\varphi\circ q$ is a $(\G,X)$-boundary map on ${A}$ and $J\subset I_{\psi}$.
\end{proof}

\begin{proposition}\label{prop:ideal}
Let $X$ be a minimal locally compact $\G$-space. Given $(\pi,\rho)\in\Rep(\G, X)$ and a $(\G, X)$-boundary map $\psi$ on $C^*_{\pi\times\rho}(\G,X)$, we have that $I_\psi$ is an ideal of $C^*_{\pi\times\rho}(\G,X)$.

In particular, if $C^*_{\pi\times\rho}(\G,X)$ admits a unique boundary map $\psi$, then $C^*_{\pi\times\rho}(\G,X)/I_\psi$ is the unique non-zero simple quotient of $C^*_{\pi\times\rho}(\G,X)$.
\end{proposition}

\begin{proof}
The left ideal $I_\psi$ is a right ideal because $\psi$ is $\G$-equivariant, $\rho(C_0(X))$ is in the multiplicative domain of $\psi$ and $C^*_{\pi\times\rho}(\G,X)$ is generated by elements of the form $\rho(f)\pi(g)$, for $f\in C_0(X)$ and $g\in\G$.
\end{proof}

%%%%%%%%%%
%%%%%%%%%%

Because of the next result,  we assume in the rest of Section~\ref{sec:bdm} and in Section~\ref{sec:simpl} that $\G$ is a \emph{countable} group. Notice that, if the $\G$-spaces in Lemma~\ref{lem:int} are compact, then it is not necessary to assume that $\G$ is countable (\cite[Lemma 3.2]{BKKO}).  

\begin{lemma}\label{lem:int}
Let $\G$ be a countable group. A continuous equivariant map $\pi\colon Y\to X$ between minimal locally compact $\G$-spaces sends sets of non-empty interior to sets of non-empty interior.
\end{lemma}

\begin{proof}
Let $U\subset Y$ be a set of non-empty interior and take $V\subset U$ a non-empty open subset such that $\overline{V}\subset U$ and $\overline{V}$ is compact. By minimality,  we have that $Y=\bigcup_{g\in\G}gV$, hence $X=\bigcup_{g\in\G}g\pi(\overline{V})$.  Since $\G$ is countable, we conclude that $\pi(\overline{V})\subset\pi(U)$ has non-empty interior.
\end{proof}

\begin{prop}\label{supp}
Let $X$ be a minimal locally compact $\G$-space, $(\pi,\rho)\in\Rep(\G,X)$, and $\psi$ a $(\G,\X)$-boundary map on $C^*_{\pi\times\tilde{\rho}}(\G,\X)$. Then for every $g\in\G$ we have  
$$\supp\psi(\pi(g))\subset \overline{\b_{\hspace{-0.1em} X}^{-1}(\interior X^g)}^{\partial(\G,\X)}.$$
\end{prop}

\begin{proof}
Fix $g\in \G$. For the sake of contradiction, assume that there is $y\in \partial(\G,\X)$ such that
$y\notin \overline{\b_{\hspace{-0.1em} X}^{-1}(\interior X^g)}^{\partial(\G,\X)}$ and  $\psi(\pi(g))(y)\neq 0$.  Since $\partial(\G,X)$ is dense in $\partial(\G,\X)$, we may assume $y\in\partial(\G,X)$. Then there is an open neighborhood $U\subset\partial(\G,X)$ of $y$ such that $U\cap \b_{\hspace{-0.1em} X}^{-1}(\interior X^g)=\emptyset$ and 
\begin{equation}\label{eq:piz}
\psi(\pi(g))(z)\neq 0
\end{equation}
 for any $z\in U$.

By Lemma~\ref{lem:int}, $\b_{\hspace{-0.1em} X}(U)$ has non-empty interior. Since $$\b_{\hspace{-0.1em} X}(U)\cap\interior X^g=\emptyset,$$ there is $u\in U$ such that $\b_{\hspace{-0.1em} X}(u)\notin X^g$. Let $f\in C_0(X)$ such that $0\leq f\leq 1$, $f(\b_{\hspace{-0.1em} X}(u))=1$ and $f(g^{-1}\b_{\hspace{-0.1em} X}(u))=0$.

Notice that $\psi(\rho(f))(u)=f(\b_{\hspace{-0.1em} X}(u))=1$ and $$\psi(\pi(g)\rho(f)\pi(g)^*)(u)=\psi(\rho(f))(g^{-1}u)=f(g^{-1}\b_{\hspace{-0.1em} X}(u))=0.$$
Applying \cite[Lemma 2.2]{HartKal} to the state $\delta_u\circ\psi$, we conclude that $\psi(\pi(g))(u)=0$, which contradicts \eqref{eq:piz}.
\end{proof}

Before we proceed to our main result on uniqueness of boundary maps and its applications, let us prove as a first application of boundary maps the following generalization of the fact that a group has the unique trace property if and only if its Furstenberg boundary action is faithful if and only if its amenable radical is trivial (\cite{BKKO}, \cite{Fur}). 

\begin{theorem}\label{faithful}
Let $X$ be a minimal compact $\G$-space. The following conditions are equivalent:
\begin{enumerate}
\item[(i)] The action $\G\act\partial(\G,X)$ is faithful.
\item[(ii)] The canonical conditional expectation $\psi\colon C(X)\rtimes_r\G\to C(X)$ is the unique equivariant conditional expectation such that  $\psi(C^*_r(\G))\subset\C1_{C(X)}$.
\item[(iii)] For any $x\in X$, the stabilizer $\G_x$ does not contain any non-trivial amenable normal subgroup of $\G$.
\item[(iv)] There is $x\in X$ such that $\G_x^0$ does not contain any non-trivial amenable normal subgroup of $\G$.
\end{enumerate}
\end{theorem}
\begin{proof}
(i)$\implies$(ii): Let $\varphi\colon C(X)\rtimes_r\G\to C(X)$ be an equivariant conditional expectation such that $\varphi(C^*_r(\G))\subset\C1_{C(X)}$.  By $\G$-injectivity, we can extend $\varphi$ to a ucp equivariant map $\tilde{\varphi}\colon C(\partial(\G,X))\rtimes_r\G\to C(\partial(\G,X))$. Furthermore, by rigidity, we have that $\tilde{\varphi}$ is a conditional expectation.  

Fix $g\in\G\setminus\{e\}$. Since $\G\act\partial(\G,X)$ is faithful,  we have that 
$$\partial(\G,X)^g\subsetneq\partial(\G,X).$$
 Note that, since $\partial(\G,X)$ is $\G$-injective, we have $\partial(\G,\partial(\G,X))=\partial(\G,X)$ and Proposition~\ref{supp} applied to the minimal $\G$-space $\partial(\G,X)$ implies that 
$$\supp\tilde{\varphi}(\delta_g)\subset\partial(\G,X)^g\subsetneq\partial(\G,X).$$
 Hence, $\varphi(\delta_g)=0$. Since $g$ was arbitrary, we conclude that $\varphi$ is the canonical conditional expectation.

(ii)$\implies$(iii): If $N$ is an amenable non-trivial normal subgroup of $\G$ such that $N\leq\G_x$ for some $x\in X$, then $(\la_{\G/N},\cP_x)\prec(\la_\G,\cP_x)$ by Proposition~\ref{prop:prec}. Hence, there is a canonical quotient $\theta\colon C(X)\rtimes_r\G\twoheadrightarrow C(X)\rtimes_r\frac{\G}{N}$. Let $E_N\colon C(X)\rtimes_r\frac{\G}{N}\to C(X)$ be the canonical conditional expectation. Then $E_N\circ\theta$ is not the canonical conditional expectation on $C(X)\rtimes_r\G$ and $E_N\circ\theta(C^*_r(\G))\subset\C 1_{C(X)}$.

(iii)$\implies$(iv) is evident.

(iv)$\implies$(i) follows from the fact that $\ker(\G\act\partial(\G,X))\subset\G_x^0$ for any $x\in X$, and is amenable by \cite[Proposition 3.3(i)]{Kaw}.
\end{proof}

\begin{corollary}
Let $X$ be a minimal compact $\G$-space such that $\G_x^0$ is amenable for some $x\in X$. Then $X$ is faithful if and only if $\partial(\G,X)$ is faithful.
\end{corollary}

\begin{proof}
Suppose $X$ is not faithful. Since $\ker(\G\act X)\leq\G_x^0$ is amenable, it follows from Theorem~\ref{faithful} that $\partial(\G,X)$ is not faithful. The other direction is immediate.
\end{proof}

The following theorem generalizes \cite[Theorem 4.4]{KalSca20}, where the same conclusion was proved for boundary actions, i.e. for minimal strongly proximal compact $\G$-spaces $X$.
\begin{theorem}\label{thm:uniq}
Let $X$ be a minimal locally compact $\G$-space and $(\pi,\rho)\in\Rep(\G, X)$ a germinal representation. Then there is a unique $(\G,X)$-boundary map $\psi$ on $ C^*_{\pi\times\rho}(\G,X)$. Furthermore, $\psi$ is given by
\[
\psi(\rho(f)\pi(g)) = (f\circ\b_X)\,\cdot\, \one_{\overline{\b_X^{-1}(\interior X^g)}}
\]
for every $g\in\G$ and $f\in C_0(X)$.
\end{theorem}

\begin{proof}
Take $f\in C_0(X)$ and $g\in\G$.  Given $y\in \b_{\hspace{-0.1em} X}^{-1}(\interior X^g)$, take $h\in C_0(X)$ such that $h(\b_{\hspace{-0.1em} X}(y))=1$, and $\supp h\subset X^g$. 

Since $(\pi,\rho)$ is a germinal representation, we have that $$\psi(\rho(f){\pi}(g)){\psi}(\rho(h))=\psi(\rho(f)){\psi}(\rho(h)).$$ Evaluating both sides of this equation at $y$, we conclude that $$\psi(\rho(f)\pi(g))(y)=\psi(\rho(f))(y)=f(\b_X(y)).$$ This implies that 
$$\psi(\rho(f)\pi(g))|_{\overline{\b_X^{-1}(\interior X^g)}}=(f\circ\b_X)|_{\overline{\b_X^{-1}(\interior X^g)}}.$$
Finally, by considering the extension $\tilde{\psi}$ of $\psi$ to $C^*_{\pi\times\tilde{\rho}}(\G,X)$ as in Proposition~\ref{prop:ext} and applying Proposition~\ref{supp}, we conclude that $\psi(\rho(f)\pi(g))$ vanishes on $\partial(\G,X)\setminus\overline{\b_{\hspace{-0.1em} X}^{-1}(\interior X^g)}$.
\end{proof}

\section{Simplicity of $C^*$-algebras of covariant representations}\label{sec:simpl}

In this section, we investigate simplicity of $C^*$-algebras generated by certain covariant representations coming from stabilizers. In particular, we generalize several $C^*$-simplicity results proved in \cite[Section 5]{KalSca20}, where the case of boundary actions was considered.
Throughout the section, we assume that $\G$ is a countable group.

\begin{proposition}\label{prop:la}
Let $X$ be a minimal locally compact $\G$-space. Given $x\in X$ and $y\in\b_X^{-1}(x)$, we have that $\La:=\{g\in\G:y\in\overline{\b_X^{-1}(\interior X^g)}\}$ is a subgroup of $\G$ such that $\G_x^0\leq\La\leq\G_x$ and $C^*_{\la_{\G/\La}\times\rho_x}(\G,X)$ is simple.
\end{proposition}
\begin{proof}
Let us first prove that $\G_x^0\subset\La\subset\G_x$. Given $g\in\G_x^0$, we have that $x\in\interior X^g$, hence $y\in\b_X^{-1}(\interior X^g)$. Therefore, $\G_x^0\subset\La$. Given $g\in\La$, there exists a net $(z_i)\subset\b_X^{-1}(\interior X^g)$ such that $z_i\to y$. Then $gx=g\b_X(y)=\lim g\b_X(z_i)=\lim\b_X(z_i)=x$. Hence, $\La\subset\G_x$.

Let $(\pi,\rho)$ be a germinal representation of $(\G,X)$. By Theorem~\ref{thm:uniq},  $C^*_{\pi\times\tilde{\rho}}(\G,\X)$ admits a unique $(\G,\X)$-boundary map $\tilde{\psi}$, and $\tilde{\psi}(\pi(g))=\one_{\overline{\b_X^{-1}(\interior X^g)}^{\partial(\G,\X)}}$ for every $g\in\G$. The positive-definite function $\delta_y\circ\tilde{\psi}|_\G$ coincides with $\one_\La$, and therefore, 
$\La$ is a subgroup. Hence, $(\la_{\G/\La},\rho_x)$ is a germinal representation by Example~\ref{ex:grprep}. Moreover, since the GNS representation associated with $\delta_y\circ\psi$ is unitarily equivalent to $\la_{\G/\La}\times\rho_x$, we conclude that $(\pi,\rho)\prec(\la_{\G/\La},\rho_x)$. 

Since $(\pi,\rho)$ was an arbitrary germinal representation, and any $(\sigma,\varphi)\in\Rep(\G,X)$ such that $(\sigma,\varphi)\prec(\la_{\G/\La},\rho_x)$ is a germinal representation by Lemma~\ref{lem:wcont-grprep}, we conclude that $C^*_{\la_{\G/\La}\times\rho_x}(\G,X)$ is simple.
\end{proof}

\begin{corollary}\label{cor:os}
Let $X$ be a minimal locally compact $\G$-space. Then $C^*_{\la_{\G/\G_x^0}\times\rho_x}(\G,X)$ is simple for every $x\in\Xoo$.
\end{corollary}
\begin{proof}
Fix $y\in\b_X^{-1}(x)$. Given $g\in\G$ such that $y\in\overline{\b_X^{-1}(\interior X^g)}$, take a net $(z_i)\subset\b_X^{-1}(\interior X^g)$ such that $z_i\to y$. Since $x\in\Xoo$, we have that $\G_{\b_x(z_i)}^0\to\G_x^0$. Moreover, since $g\in\G_{\b_X(z_i)}^0$ for every $i$, it follows that $g\in\G_x^0$. Therefore, we conclude $\G_x^0=\{g\in\G:y\in\overline{\b_X^{-1}(\interior X^g)}\}$ and the result follows from Proposition~\ref{prop:la}. 
\end{proof}
\begin{remark}
Corollary \ref{cor:os} can also be deduced from the results in \cite{KM21} in the following way: Let $\cG$ be the groupoid of germs associated to $\G\act X$.  It follows from \cite[Theorem 7.26]{KM21} that the essential $C^*$-algebra $C^*_{\mathrm{ess}}(\cG)$ is simple.  Moreover, for each $x\in X$,  we have that $\cG_x$ can be identified with $\G/\G_x^0$ and, if $x\in \Xoo$, by arguing as in \cite[Remark 9.3]{KKLRU}, we get that there is a surjective $*$-homomorphism $C^*_{\mathrm{ess}}(\cG)\to C^*_{\la_{\G/\G_x^0}\times\rho_x}(\G,X)$. 
\end{remark}

\begin{remark}
In \cite[Theorem 7.7]{KalSca20}, the authors showed that, given a $\G$-boundary $X$ with Hausdorff germs (i.e., such that $\Xoo=X$) and a germinal representation $(\pi,\rho)$ of $(\G,X)$, then $(\pi,\rho)$ weakly contains the regular representation of the groupoid of germs of the action (which is given by $\bigoplus_{x\in X}(\la_{\G/\G_x^0},\rho_x)$). Corollary~\ref{cor:os} and the proof of Proposition~\ref{prop:la} imply that it suffices to assume that $X$ is a minimal compact $\G$-space with Hausdorff germs (i.e., $\G\act X$ does not necessarily have to be a boundary action).
\end{remark}

\begin{lemma}\label{aux}
Let $X$ be a minimal locally compact $\G$-space. Given $x\in X$ and $\G_x^0\leq\La\leq\G_x$, if $C^*_{\la_{\G/\La}\times\rho_x}(\G,X)$ is simple, then $\frac{\La}{\G_x^0}$ is amenable.
\end{lemma}
\begin{proof}
We denote by $\rho_x^0\colon C_0(X)\to B(\ell^2(\G/\G_x^0))$ and $\rho_x\colon C_0(X)\to B(\ell^2(\G/\La))$ the maps defined as in \eqref{Poisson}.

If $C^*_{\la_{\G/\La}\times\rho_x}(\G,X)$ is simple, then by Proposition~\ref{prop:ideal} the unique boundary map on $C^*_{\la_{\G/\La}\times\rho_x}(\G,X)$ is faithful. Now, arguing as in \cite[Lemma 5.5]{KalSca20}, we conclude that $(\la_{\G/\La},\rho_x)\prec(\la_{\G/\G_x^0},\rho_x^0)$, hence $\frac{\La}{\G_x^0}$ is amenable by Proposition~\ref{prop:prec}.
\end{proof}

\begin{theorem}\label{thm:gk}
Let $X$ be a minimal locally compact $\G$-space. Given $x\in X$, the $C^*$-algebra $C^*_{\la_{\G/\G_x}\times\rho_x}(\G,X)$ is simple if and only if $\frac{\G_x}{\G_x^0}$ is amenable.
\end{theorem}
\begin{proof}
The forward implication follows from Lemma~\ref{aux}. 

Conversely, assume $\frac{\G_x}{\G_x^0}$ is amenable. Take $\La$ as in Proposition~\ref{prop:la}. Since $\frac{\G_x}{\G_x^0}$ is amenable, we have that $\La$ is co-amenable in $\G_x$ and, $(\la_{\G/\G_x},\rho_x)\prec(\la_{\G/\La},\rho_x)$ by Proposition~\ref{prop:prec}. Therefore, $C^*_{\la_{\G/\G_x}\times\rho_x}(\G,X)$ is simple.
\end{proof}

For the following statement we refer the reader to \cite{Anan02} for the definition of topologically amenable actions.
\begin{corollary}
Let $X$ be a topologically amenable locally compact minimal $\G$-space.
Then $C^*_{\la_{\G/\G_x}\times\rho_x}(\G,X)$ is simple for every $x\in X$.

In particular, for any minimal action $\G\act X$ of an amenable group $G$ on a locally compact space $X$, and any $x\in X$, $C^*_{\la_{\G/\G_x}\times\rho_x}(\G,X)$ is simple.
\end{corollary}

\begin{proof}
If $\G\act X$ is topologically amenable, then $\G_x$ is amenable for every $x\in X$. Thus the claim follows from Theorem~\ref{thm:gk}.
\end{proof}

\begin{remark}
This result generalizes \cite[Theorem 8.1]{Kaw}, in which Kawabe showed that, given a minimal compact $\G$-space $X$ and $y\in\partial(\G,X)$, it holds that $C^*_{\la_{\G/\G_y}\times\rho_y}(\G,\partial(\G,X))$ is simple. Indeed, this statement is a consequence of Theorem~\ref{thm:gk}, since $\G_y=\G_y^0$ for every $y\in \partial(\G,X)$ (Proposition~\ref{Kaw3.3}).
\end{remark}

\begin{example}
Recall that Thompson's group $F$ is the group of homeomorphisms on $(0,1)$ which are piecewise linear, with finitely many breakpoints, all lying in $\Z[1/2]\cap(0,1)$, and whose slopes are integer powers of $2$. Clearly, $F\act (0,1)$ is minimal. 

Given $x\in \Z[1/2]\cap(0,1)$, one can readily check that $\frac{F_x}{F_x^0}\simeq \Z\times\Z$ (where the identification involves taking $\log_2$ of the left and right derivatives at $x$ of an element $g\in F_x$). Hence,  by Theorem~\ref{thm:gk},  we obtain that $C^*_{\la_{F/F_x}\times\rho_x}(F,(0,1))$ is simple.
\end{example}

\subsection*{$C^*$-irreducible inclusions}
In his recent paper \cite{Ror-C*-irr} R\o{}rdam introduced and studied the notion of irreducible inclusions of $C^*$-algebras, namely inclusions of unital simple $C^*$-algebras with the property that all intermediate $C^*$-algebras are also simple. Among the canonical examples of such inclusions are the inclusions arising from crossed products of dynamical systems. 

Our results provide more examples of such inclusions (see also \cite[Theorem 1.3]{AmruKal} and \cite[Theorem 1.5]{AmrUrs}).

In order to apply our results, we first draw the following conclusion of Theorem~\ref{thm:uniq}.

\begin{theorem}\label{thm:C*irr}
Let $Y$ be a minimal compact $\G$-space and $C(Y)\subset C(X) \subset C(\partial(\G, Y))$ be an inclusion of $\G$-$C^*$-algebras.  Given a germinal representation $(\pi,\rho)$ such that $ C^*_{\pi\times\rho}(\G,X)$ is simple, we have that the inclusion $C^*_{\pi\times\rho}(\G,Y)\subset C^*_{\pi\times\rho}(\G,X)$ is $C^*$-irreducible.  
\end{theorem}

\begin{proof}
First note that $\partial(\G, X) = \partial(\G, Y)$ canonically. By Theorem~\ref{thm:uniq}, there is a unique $(\G, X)$-boundary map on $C^*_{\pi\times\rho}(\G,X)$, which is faithful by Proposition~\ref{prop:ideal}. Let $C^*_{\pi\times\rho}(\G,Y) \subset A \subset C^*_{\pi\times\rho}(\G,X)$ be an intermediate $C^*$-algebra. Then $A$ is in particular a $(\G, Y)$-$C^*$-algebra. Given a $(\G,Y)$-boundary map $\psi$ on $A$, by $\G$-injectivity, we may extend $\psi$ to a $\G$-equivariant ucp map $\tilde\psi: C^*_{\pi\times\rho}(\G,X) \to C(\partial(\G, X))$. Since $\psi(\iota(f))=f\circ\b_Y$ for every $f\in C(Y)$, then by rigidity, $\psi(\iota(f))=f\circ\b_X$ for every $f\in C(X)$, and so $\tilde\psi$ is a $(\G,X)$-boundary map, hence faithful. This implies by Proposition~\ref{prop:ideal1} that $A$ is simple.
\end{proof}

The following is a consequence of Corollary~\ref{cor:os} and Theorems~\ref{thm:gk} and~\ref{thm:C*irr}.
\begin{thm}
Let $Y$ be a minimal compact $\G$-space and $C(Y)\subset C(X) \subset C(\partial(\G, Y))$ be an inclusion of $\G$-$C^*$-algebras. Then the following inclusions are $C^*$-irreducible:
\begin{enumerate}
\item[(i)]For every $x\in\Xoo$,  the inclusion $$C^*_{\la_{\G/\G_x^0}\times\rho_x}(\G,Y) \subset C^*_{\la_{\G/\G_x^0}\times\rho_x}(\G,X).$$
\item[(ii)]For every $x\in X$ such that $\frac{\G_x}{\G_x^0}$ is amenable, the inclusion $$C^*_{\la_{\G/\G_x}\times\rho_x}(\G,Y) \subset C^*_{\la_{\G/\G_x}\times\rho_x}(\G,X).$$
\end{enumerate}
\end{thm}

\section{Simplicity of crossed products}

In this section, we investigate simplicity of crossed products associated to actions on locally compact spaces.

A locally compact $\G$-space $X$ is said to have the \emph{intersection property} if for any non-zero ideal $I\unlhd C_0(X)\rtimes_r\G$ we have that $I\cap C_0(X)\neq\{0\}$.

\begin{lem}\label{lem:intp} Let $Y$ be a locally compact $\G$-space and $X$ an open $\G$-invariant subset of $Y$.  Consider the following properties.
\begin{enumerate}
\item[(i)]The space $Y$ has the intersection property.
\item[(ii)] The space $X$ has the intersection property.
\end{enumerate}
Then (i)$\implies$(ii).  If $X$ is dense in $Y$, then (ii)$\implies$(i).
\end{lem}

\begin{proof}
The fact that (i)$\implies$(ii) follows immediately from the fact that $C_0(X)\rtimes_r\G$ is an ideal of $C_0(Y)\rtimes_r\G$.

Assume $X$ has the intersection property and is dense in $Y$; we will show that also $Y$ has the intersection property.  

Suppose there exists a non-zero ideal $I\unlhd C_0(Y)\rtimes_r\G$ such that $I\cap C_0(Y)=\{0\}$. In particular, $I\cap C_0(X)=\{0\}$. Since $X$ has the intersection property, it implies 
\begin{equation}\label{eq:ideal}
I\cap (C_0(X)\rtimes_r\G)=\{0\}.
\end{equation}  
Let $E\colon C_0(Y)\rtimes_r \G\to C_0(Y)$ be the canonical conditional expectation.  Since $E$ is faithful and $X$ is dense in $Y$, there is $a\in I$ and $f\in C_0(X)$ such that $E(af)=E(a)f\neq 0$, which contradicts \eqref{eq:ideal}.
\end{proof}

Denote by $\mathcal{S}_a(\G)$ the set of amenable subgroups of $\G$, endowed with the Chabauty topology. Given a locally compact $\G$-space $X$, let $\mathcal{S}_a(X,\G):=\{(x,\La)\in X\times\mathcal{S}_a(\G):\La\leq\G_x\}$.  The group $\G$ acts on $\mathcal{S}_a(X,\G)$ by $g(x,\La):=(gx,g\La g^{-1})$, for $g\in\G$ and $(x,\La)\in\mathcal{S}_a(X,\G)$. Since $\mathcal{S}_a(\G)$ is compact, one can readily check that the map $p_X\colon \mathcal{S}_a(X,\G)\to X$ given by projecting on the first coordinate is closed.   

The following result extends some results from \cite{Kaw} which were shown in the case that the space is compact.
\begin{theorem}\label{thm:intp}
Let $X$ be a locally compact $\G$-space. The following conditions are equivalent:
\begin{enumerate}
\item[(i)]The space $X$ has the intersection property.
\item[(ii)] For every closed $\G$-invariant set $Y\subset \mathcal{S}_a(X,\G)$ such that $p_X(Y)=X$, the space $Y$ contains $X\times\{e\}$.
\item[(iii)] The crossed product $C_0(X)\rtimes_r\G$ admits a unique $(\G,X)$-boundary map.
\item[(iv)] The action $\G\act \partial(\G,X)$ is topologically free.
\end{enumerate}
\end{theorem}
\begin{proof}
If $X$ is compact, then the result follows from \cite[Theorems 3.4, 4.2 and 5.2]{Kaw}. So we assume that $X$ is not compact.

(i)$\implies$(ii). By Lemma~\ref{lem:intp}, $\X$ also has the intersection property. 

Let $Y$ be a closed $\G$-invariant subset of $\mathcal{S}_a(X,\G)$ such that $p_X(Y)=X$. Note that $W:=Y\cup \mathcal{S}_a(\X\setminus X,\G)$ is a closed $\G$-invariant subset of $\mathcal{S}_a(\X,\G)$ such that $p_{\X}(W)=\X$.  It follows from \cite[Theorem 5.2]{Kaw} that $\X\times\{e\}\subset W$, which implies that $X\times\{e\}\subset Y$.

(ii)$\implies$(i).  Let us show that $\X$ has the intersection property. By \cite[Theorem 5.2]{Kaw}, it suffices to show that, given $Y\subset \mathcal{S}_a(\X,\G)$ closed and $\G$-invariant such that $p_{\X}(Y)=Z$, the space $Y$ contains $\X\times\{e\}$.  Let $W:=Y\cap\mathcal{S}_a(X,\G)$. Since $p_X(Y)=X$, we conclude that $X\times\{e\}\subset W$. Since $X$ is dense in $\X$, it follows that $\X\times\{e\}\subset Y$.

(i)$\iff$(iii). This follows immediately from the corresponding fact for compact $\G$-spaces and Proposition~\ref{prop:ext}.  

(iv)$\iff$(i). This follows immediately from the corresponding fact for compact $\G$-spaces.
\end{proof}

The following result was proven in \cite[Theorem 6.1]{Kaw} for actions on compact spaces.
\begin{theorem}\label{thm:scp}
Let $X$ be a minimal locally compact $\G$-space. The following conditions are equivalent:
\begin{enumerate}
\item[(i)] The reduced crossed product $C_0(X)\rtimes_r\G$ is simple.
\item[(ii)] For every $x,z\in X$ and every amenable subgroup $\La$ of $\G_x$, there exists a net $(g_i)\subset\G$ such that $g_i\La g_i^{-1}\to\{e\}$ and $g_ix\to z$. 
\item[(iii)]There exists $x\in X$ such that for every amenable subgroup $\La$ of $\G_x$, there exists a net $(g_i)\subset\G$ such that $g_i\La g_i^{-1}\to\{e\}$ and $g_i x$ converges to a point in $X$.
\end{enumerate}
\end{theorem}
\begin{proof}
Since $X$ is minimal, we have that $C_0(X)\rtimes_r\G$ is simple if and only if $X$ has the intersection property.

(i)$\implies$(ii). Given $x\in X$ and $\La\leq\G_x$ amenable, let $Y$ be the closure of the $\G$-orbit of $(x,\La)$ in $\mathcal{S}_a(\G,X)$. Since $p_X$ is a closed map and $X$ is minimal, we have $p_X(Y)=X$. By Theorem~\ref{thm:intp}, we conclude that $X\times\{e\}\subset Y$. Hence, there exists a net $(g_i)\subset\G$ such that $g_i\La g_i^{-1}\to\{e\}$ and $g_ix\to z$.

 (ii)$\implies$(iii) is immediate.

(iii)$\implies$(i).  We will use Theorem~\ref{thm:intp} to show that $X$ has the intersection property. Take $Y\subset \mathcal{S}_a(\G,X)$ closed and $\G$-invariant such that $p_X(Y)=X$.  Let $\La\leq\G_x$ amenable such that $(x,\La)\in Y$ and $(g_i)\subset\G$ a net such that $g_i\La g_i^{-1}\to\{e\}$ and $g_ix$ converges to a point in $X$.  In particular, $(X\times\{e\})\cap Y\neq\emptyset$. Since $X$ is minimal, we conclude that $X\times\{e\}\subset Y$.
\end{proof}
\begin{remark}
In \cite[Section 8]{CryNag}, Crytser and Nagy obtained some partial results about the simplicity of reduced crossed products of locally compact $\G$-spaces. 

Shortly after this paper appeared on arXiv, Kennedy, Kim, Li, Raum and Ursu obtained in \cite{KKLRU} a characterization of the intersection property for étale groupoids which generalizes Theorem~\ref{thm:intp}.
\end{remark} 
\begin{example}Given a $C^*$-simple group $\G$ and a minimal compact $\G$-space $X$, it follows from \cite[Theorem 7.1]{BKKO} that $C(X)\rtimes_r\G$ is simple. This is not necessarily true for actions on locally compact spaces.
Indeed, given a non-trivial group $\G$, take $\La$ a non-trivial amenable subgroup of $\G$. By \cite[Corollary 2.9]{Gre}, $c_0(\G/\La)\rtimes_r\G$ is isomorphic to $K(\ell^2(\G/\La))\otimes C^*_r(\La)$, hence not simple (see also Example~\ref{ex:amenable}).
\end{example}

The following result was proven in the compact case in \cite[Lemma 2.4]{Kaw}.

\begin{proposition}\label{prop:conv}
Let $X$ be a locally compact $\G$-space.  
\begin{enumerate}
\item[(i)] If $\{x\in X:\text{ $\G_x$ is $C^*$-simple}\}$ is dense in $X$, then $X$ has the intersection property.
\item[(ii)] If $X$ has the intersection property and $\G_x^0$ is amenable for all $x\in X$, then $X$ is topologically free.
\end{enumerate}
\end{proposition}
\begin{proof}
(i). This is an immediate consequence of \cite[Lemma 2.4.(i)]{Kaw} and Lemma~\ref{lem:intp}.

(ii).  Let $(\pi,\rho):=\bigoplus_{x\in X}(\la_{\G/\G_x^0},\rho_x)$. Since each $\G_x^0$ is amenable, there is a canonical $*$-homomorphism $\varphi\colon C_0(X)\rtimes_r\G\to C^*_{\pi\times\rho}(\G,X)$. If $X$ is not topologically free, then, since $(\pi,\rho)$ is a germinal representation, it follows immediately that $\varphi$ is not injective. Since $\ker \varphi\cap C_0(X)=\{0\}$, we conclude that $X$ does not have the intersection property.
\end{proof}
\begin{example}
Let $F$ be Thompson's group.  It is known that $F$ is non-amenable if and only if it is $C^*$-simple (\cite[Corollary 4.2]{LeM}).  

Clearly, $F\act(0,1)$ is not topologically free.  
It follows from \cite[Lemma 4.4]{CFP} that, for any $x\in \Z[1/2]\cap(0,1)$,  it holds that $F_x\simeq F\times F$. Therefore, we conclude from Proposition~\ref{prop:conv} that $F$ is non-amenable if and only if $C_0((0,1))\rtimes_r F$ is simple.
\end{example}

%%%%%%%%%

\end{document}